\documentclass[12 pt]{article}

\usepackage{graphicx}
\usepackage{latexsym,amsmath,amsfonts,amsthm,amsxtra,framed,url}
\usepackage{comment}
\usepackage{float}
\usepackage{amssymb}
\usepackage[english]{babel}
\usepackage[small]{caption}
\usepackage{algorithm}
\usepackage{algpseudocode}

\algnewcommand\OR{\textbf{or}}
\algnewcommand\AND{\textbf{and}}

\makeindex
   \newtheorem{theorem}{Theorem}[section]
   
   \newtheorem{lemma}[theorem]{Lemma}
   \newtheorem{corollary}[theorem]{Corollary}

   \newtheorem{definition}[theorem]{Definition}
   \newtheorem{example}[theorem]{Example}

   \newtheorem{remark}[theorem]{Remark}

\begin{document}

\title{Invariance Conditions for Nonlinear Dynamical Systems}

\author{Zolt\'{a}n Horv\'{a}th\thanks{Department of Mathematics and Computational Sciences, Sz\'{e}chenyi Istv\'{a}n
University, 9026 Gy\H{o}r, Egyetem t\'{e}r 1, Hungary, Email: horvathz@sze.hu.}
\and Yunfei Song\thanks{Department of Industrial and Systems
Engineering, Lehigh University, 200 West Packer Avenue, Bethlehem,
PA, 18015-1582. Email: yus210@lehigh.edu.}
\and Tam\'{a}s
Terlaky\thanks{Department of Industrial and Systems Engineering,
Lehigh University, 200 West Packer Avenue, Bethlehem, PA,
18015-1582. Email: terlaky@lehigh.edu.}}

\maketitle

\begin{abstract}
Recently, Horv\'{a}th, Song, and Terlaky [\emph{A novel unified approach to invariance condition of dynamical system, 
submitted to Applied Mathematics and Computation}] proposed a novel unified approach to study, i.e., invariance conditions, sufficient and necessary conditions, 
 under which some convex sets are invariant sets for linear dynamical systems. 

In this paper, by utilizing analogous methodology,  we generalize the results for nonlinear dynamical systems. 
First, the Theorems of Alternatives, i.e., the nonlinear Farkas lemma and the \emph{S}-lemma, together with 
Nagumo's Theorem  are utilized to derive invariance conditions for discrete and continuous systems. 
Only  standard assumptions are needed to establish invariance of  broadly used convex sets, 
including polyhedral and ellipsoidal sets. 
Second, we  establish an optimization framework to computationally verify the derived invariance conditions. Finally, we derive analogous  invariance conditions  without any conditions. 

\end{abstract}

\textbf{Keywords:} Invariant Set, Nonlinear Dynamical System, Polyhedral Set, Ellipsoid, Convex Set

\section{Introduction}
Positively invariant set is an important concept, and it has a wide range of applications in dynamical systems and control theory, see, e.g., \cite{Blanchini, boyd, luen, jinlai}. Let a state space and a dynamical system be given. A subset $\mathcal{S}$ in the state space is called a \emph{positively invariant set} of the dynamical system if any forward trajectory originated from $\mathcal{S}$ stays in $\mathcal{S}$. For the sake of simplicity, throughout the paper we use  \emph{invariant set} to refer to positively invariant set. Some classical examples of invariant sets are equilibria, limit cycles, etc (see \cite{tesc}). In higher dimensional spaces,  examples of invariant sets are e.g., invariant torus and chaotic attractor, see, e.g., \cite{tesc}. A similar concept in dynamical system is stability, which  is one of the most commonly studied property of invariant sets. Intuitively, an invariant set is referred to as stable if any trajectories starting close to it remain close to it, and unstable if they do not. Positively invariant set  is an important concept and an efficient tool for the design of controllers of constrained systems. For example, for a given closed-loop control system, the state and control constraints hold when the initial state is chosen from a certain positively invariant set, see, e.g., \cite{zhou}.

A fundamental question  is to develop efficient tools to verify if a given set is an invariant set for a given (discrete or continuous) dynamical system. Sufficient and necessary conditions under which a set is an invariant set for a dynamical system are important both from the theoretical and practical aspects. Such sufficient and necessary conditions are usually referred to as \emph{invariance conditions}, see, e.g., \cite{song1}. Invariance conditions can be considered as  special tools to study the relationship between the invariant set and the dynamical system. 
They also provide  alternative ways to design efficient algorithms to construct  invariant sets. 
Linear discrete and continuous  dynamical systems have been extensively  studied  in recent decades, since such systems have a wide range of applications in control theory, see, e.g., \cite{alil, bits2, krein}. Invariance condition for linear systems are relatively easy to derive while analogous conditions for nonlinear systems are more difficult to derive. 
 Convex sets are often chosen as candidates for invariant sets of linear dynamical systems. These sets include polyhedron, see, e.g., \cite{bits1, bits2, Blanchini}, polyhedral cone, see, e.g.,  \cite{horva2, murray}, ellipsoid, see, e.g., \cite{boyd,zhou}, and Lorenz cone, see, e.g., \cite{birkhoff, schne, stern}. Recently, a novel unified method is presented in \cite{song1} to derive invariance conditions for these classical sets   for both linear discrete  and linear continuous  dynamical systems. 
Invariant sets for nonlinear dynamical systems are more complicated to study. The localization problem of compact invariant sets for discrete nonlinear system is studied in \cite{kana}.  A simple criteria to verify if a general convex set is a robust control invariant set for a nonlinear uncertain system is presented in\cite{fiacc}.  Invariant set for discrete system is studied in \cite{kerr}, and an application to model predictive control is provided. The steplength threshold for preserving invariance of a set when applying a discretization method to continuous systems is studied in \cite{song2, song3}.

 In this paper, we present invariance conditions for some classical sets for nonlinear discrete and continuous dynamical systems. This is motivated by the fact that most problems in the real world are often  described by nonlinear systems rather than linear  systems. Therefore there is a need to investigate efficient invariance condition to verify sets to be invariant sets for a nonlinear dynamical system.  The main tools used to derive invariance conditions for discrete and continuous dynamical systems are the so called Theorems of Alternatives, e.g., Farkas lemma \cite{polik, Roos}, \emph{S}-Lemma \cite{polik, yaku}, and Nagumo Theorem \cite{Blanchini, nagu}, respectively. For each invariance condition, we also present an equivalent optimization problem, which provides the possibility to use current advanced optimization algorithms or software to verify the invariance  property of given sets. 
 
The main contribution of this paper is that we propose novel invariance conditions for general discrete and continuous systems using a novel and simple approach.  Our novel approach establishes a close connection between the theory of invariant sets and optimization theory, as well as provides the possibility of using current advanced optimization algorithms and methodology to solve invariant set problems.   


\emph{Notation and Conventions.} To avoid unnecessary repetitions,
the following notations and conventions are used in this paper. The $i$-th row of a matrix $G$ is denoted by $G_i^T.$ The interior and the boundary of a set $\mathcal{S}$ is denoted by
int$(\mathcal{S})$ and  $\partial \mathcal{S}$, respectively.
  The index set $\{1,2,...,n\}$ is
denoted by $\mathcal{I}(n).$

\section{Preliminaries}
In this paper, we consider the following discrete and continuous dynamical systems:
\begin{equation}\label{slnlsdis}
x_{k+1}=f_d(x_k),
\end{equation}
\begin{equation}\label{slnlscon}
\dot{x}(t)=f_c(x(t)),
\end{equation}
where $x_k,x\in \mathbb{R}^n$ are  \emph{state variables}, and $f_d,f_c: \mathbb{R}^n\rightarrow \mathbb{R}^n$ are continuous differentiable functions. When $f_d(x)=Ax$ (or $f_c(x)=Ax$) with $A$ being an $n$ by $n$ matrix, then (\ref{slnlsdis}) (or (\ref{slnlscon})) is a linear discrete (or continuous) dynamical system.

\begin{definition}\label{def1}
A set $\mathcal{S}$ is an  invariant  set for
the discrete system (\ref{slnlsdis}) if  $ x_k\in \mathcal{S}$
implies $ x_{k+1}\in \mathcal{S}$ for all
 $k\in \mathbb{N}$.
A set $\mathcal{S}$ is an invariant set for
the continuous system (\ref{slnlscon}) if $x(0)\in \mathcal{S}$ implies
$x(t)\in\mathcal{S}$  for all $t\geq0$.
\end{definition}

A \emph{polyhedron}\footnote{For the sake of simplicity, we assume that there exists an interior point in the polyhedron. }, denoted by $\mathcal{P}\in \mathbb{R}^n$, is represented as 
$\mathcal{P}=\{x\in \mathbb{R}^n\,|\,Gx\leq b\},
$
where $G\in \mathbb{R}^{m\times n} $ and $ b\in \mathbb{R}^m$. 
An \emph{ellipsoid}, denoted by $\mathcal{E}\in \mathbb{R}^n$, centered at the
origin is defined as 
$\mathcal{E}=\{x\in\mathbb{ R}^n \,|\, x^TQx\leq 1\},$
where  $Q\in \mathbb{R}^{n\times n}$ and $Q\succ0$. Note that any ellipsoid with nonzero center can be transformed to an ellipsoid centered at the origin, see, e.g., \cite{harris}. A set $\mathcal{S}$ is said to be \emph{convex} if $\alpha x+(1-\alpha)y\in \mathcal{S}$ for any $x,y\in \mathcal{S}$ and $\alpha\in[0,1].$ One can show that any polyhedra and ellipsoids are both convex sets.

The following nonlinear Farkas lemma \cite{polik} and \emph{S}-lemma \cite{polik, yaku}, which are also refereed to as the Theorems of Alternatives,  are extensively studied in the optimization community. In this paper, we apply these two lemmas as our  tools to derive invariance conditions of sets for discrete systems. 
\begin{theorem}\label{nlf}
\emph{\textbf{(nonlinear Farkas lemma\footnote{The Slater condition means that there exists a $\hat{y}\in \mathbb{R}^n$, such that $g_j(\hat{y})\leq0$ for all $j$ when $g_j(x)$ is linear, and $ g_j(\hat{y})<0$ for all $j$ when $g_j(x)$ is nonlinear.} \cite{polik})}} Let $h(y),g_1(y),g_2(y),...,g_m(y):\mathbb{R}^n\rightarrow\mathbb{R}$ be convex functions. Assume that the Slater condition is satisfied.  Then the following two statements are equivalent:
\begin{itemize}
\item The inequality systems $h(y)<0, ~g_j(y)\leq 0, j=1,2,...,m$ have no solution.
\item There exist $\beta_1,\beta_2,...,\beta_m\geq0$, such that $h(y)+\sum\limits_{j=1}^m\beta_jg_j(y)\geq0$ for all $y\in \mathbb{R}^n.$
\end{itemize}
\end{theorem}

\begin{theorem}\label{slemma}
\emph{\textbf{(\textit{S}-lemma \cite{polik, yaku})}}  Let
$h(y),g(y):\mathbb{R}^n\rightarrow \mathbb{R}$ be quadratic functions,
and suppose that there is a $\hat{y}\in \mathbb{R}^n$ such that
$g(\hat{y})<0$. Then the following two statements are equivalent:
\begin{itemize}
  \item The inequality system  $h(y)<0, g(y)\leq0$ has no solution.
  \item There exists a scalar $\beta\geq 0$, such that
  $h(y)+\beta g(y)\geq0,$ for all $y\in \mathbb{R}^n.$
\end{itemize}
\end{theorem}

The following Nagumo Theorem \cite{nagu} is a general theoretical result which can be considered as invariance condition of a closed and convex set for continuous systems. This theorem is chosen as a tool to derive the invariance condition of sets for continuous systems. 
\begin{theorem}\label{nagumo}
\emph{\textbf{(Nagumo  \cite{Blanchini, nagu})}} Let
$\mathcal{S}\subseteq\mathbb{R}^n$ be a closed convex set, and
assume that $\dot{x}(t)=f(x(t))$, where
$f:\mathbb{R}^n\rightarrow \mathbb{R}^n$ is a continuous function,
has a unique solution for every  $x(0)\in
\mathcal{S}$. Then $\mathcal{S}$ is an invariant set for this system
if and only if
\begin{equation}\label{cond3}
f(x)\in \mathcal{T}_\mathcal{S}(x),\text{  for all }  x\in \partial
\mathcal{S},
\end{equation}
where  $\mathcal{T}_\mathcal{S}(x)$ is the tangent cone\footnote{The tangent cone $\mathcal{T}_\mathcal{S}(x)$ is denoted as follows:  
$\mathcal{T}_\mathcal{S}(x)=\{y\in
\mathbb{R}^n\;|\;\underset{t\rightarrow0_+}{\lim\inf}\frac{{\text{dist}}(x+ty,\mathcal{S})}{t}=0\},
$
where   $\text{dist}(x,\mathcal{S})=\inf_{s\in\mathcal{S}}\|x-s\|.$} of
$\mathcal{S}$ at $x.$
\end{theorem}
The geometric interpretation of Theorem \ref{nagumo} is clear, i.e., a set $\mathcal{S}$ is an invariant set for the continuous system if and only if the tangent line of the trajectory $x(t)$ cannot point out of its tangent cone. According to \cite{Blanchini}, we have that the Nagumo Theorem cannot be extended to discrete systems.

\section{Invariance Conditions for Discrete Systems}\label{dis}
In this section, under certain assumptions we present invariance conditions of polyhedral sets, ellipsoids, and convex sets for discrete systems. The introduction of these assumptions ensures that the Theorems of Alternative can be applied to derive invariance conditions. 
First,  an invariance condition of polyhedral sets for  discrete  systems is presented as follows.
\begin{theorem}\label{ch5:polythm}
Let  a polyhedron  $\mathcal{P}=\{x\,|\,Gx\leq b\}$, where $G\in \mathbb{R}^{m\times n}$ and $b\in \mathbb{R}^m,$ and  the discrete system be given as in   (\ref{slnlsdis}). Assume  that $b_i-G_i^Tf_d(x) $  are convex functions for all $i\in\mathcal{I}(m)$. Then $\mathcal{P}$ is an invariant set for the discrete
system (\ref{slnlsdis}) if and only if there exists a matrix $H\geq0$,
such that
\begin{equation}\label{slede911}
HGx-Gf_d(x)\geq Hb-b,~~ \text{ for all } x\in \mathbb{R}^n.
\end{equation}
\end{theorem}

\begin{proof}
 We have that  $\mathcal{P}$ is an invariant set for the discrete system (\ref{slnlsdis}) if and only if $\mathcal{P}\subseteq \mathcal{P}'=\{x\,|\,Gf_d(x)\leq b\}$. The latter one means that for every $i\in\mathcal{I}(m),$ the system $G_i^Tf_d(x)>b_i$ and $Gx\leq b$  has no solution. Let us assume to the contrary that there exists an $x^*$ and $i^*$, such that $G_{i^*}^Tf_d(x^*)>b_{i^*}$ and $Gx^*\leq b.$ Then we have $x^*\in \mathcal{P}$ but $x^*\notin\mathcal{P}'$, which contradicts to $\mathcal{P}\subseteq\mathcal{P}'.$  Also, since
$b_i-G_i^Tf_d(x)$  is a convex function,
then, according to the convex Farkas Lemma  \ref{nlf}, we have that there exists a vector $H_i\geq0$ and $H_i\in \mathbb{R}^n,$ such that
\begin{equation*}
b_i-G_i^Tf_d(x)+H_i^T(Gx-b)\geq0, ~~~\text{ for all } x\in\mathbb{R}^n.
\end{equation*}
Writing $H_i^T$ for all $i\in \mathcal{I}(m)$ together into a matrix $H,$  we have  $H\geq0$ and
\begin{equation*}
b-Gf_d(x)+H(Gx-b)\geq0, ~~ \text{ for all } x\in \mathbb{R}^n,
\end{equation*}
which is the same as (\ref{slede911}).
\end{proof}


One can use algebraic method to verify if condition (\ref{slede911}) holds when $f_d(x)$ is given. The algebraic method may be very challenge. Here  we present the following optimization methodology to equivalently solve condition (\ref{slede911}).  

\begin{remark}
Consider the following  $m$ optimization problems
\begin{equation}\label{sdopt351}
\max_{H_i\geq0}\min_{x\in\mathbb{R}^n}\{H_i^TGx-G_i^Tf_d(x)-H_i^Tb+b_i\} \text{ for all } i\in\mathcal{I}(m).
\end{equation}
If the global optimal objective values of the $m$ optimization problems in (\ref{sdopt351}) are all  nonnegative, then we can claim that condition (\ref{slede911}) holds.
\end{remark}



In Theorem \ref{ch5:polythm}, we do not specifically assume that the system is a linear or a nonlinear system.  If the  system in Theorem \ref{ch5:polythm} is a linear dynamical system, then we  have the following corollary, which is an invariance condition of polyhedral sets for linear systems. Note that Corollary \ref{chp5:cor1}  can also be referred to \cite{song1}. An alternative proof for  Corollary \ref{chp5:cor1}, using optimality conditions is presented in the Appendix.

\begin{corollary}\label{chp5:cor1} \emph{\cite{bits2, song1}}
Let  a polyhedron  $\mathcal{P}=\{x\,|\, Gx\leq b\}$ where $G\in \mathbb{R}^{m\times n}$ and $b\in \mathbb{R}^m$ be given, and  the discrete system given  in   (\ref{slnlsdis}) be linear, i.e., $f_d(x)=A_dx$, where $A_d\in \mathbb{R}^{n\times n}$. Then $\mathcal{P}$ is an invariant set of the discrete
system (\ref{slnlsdis}) if and only if there exists a matrix $H\in \mathbb{R}^{m\times m}$ and\footnote{Here $H\geq0$ means that all the entries of $H$ are nonnegative.} $H\geq0$, such that $HG=GA$ and $Hb\leq b.$
\end{corollary}
\begin{proof}
Since the system is linear, $b_i-G_i^TAx$ are convex functions for all $i\in\mathcal{I}(m).$ According to Theorem
\ref{ch5:polythm}, there exists a matrix $H\geq0,$ such that condition (\ref{slede911}) holds, i.e.,
\begin{equation}\label{slede4f1}
(HG-GA)x\geq Hb-b,~~ \text{ for all } x\in \mathbb{R}^n.
\end{equation} 
Note that $(\ref{slede4f1})$ holds for all $x\in \mathbb{R}^n$. One can easily show that (\ref{slede4f1}) is equivalent to $HG=GA$ and $Hb\leq b$. The proof is complete.
\end{proof}

\textbf{In Theorem \ref{ch5:polythm}, we have the condition that $b_i-G_i^Tf_d(x)$ are convex function for all $i$. Recall that a function, which is twice differentiable, is convex if and only if its Hessian is positive semidefinite for all $x$. Thus, to verify  $b_i-G_i^Tf_d(x)$ are convex, one can check if $G_i^T\nabla^2 f(x)\preceq0$ for all $x\in \mathbb{R}^n.$ }
We now give an example to illustrate Theorem \ref{ch5:polythm}. 
\begin{example}
Let the discrete system be given as $\xi_{k+1}=-\xi_{k}+2\eta_{k}-\xi_{k}^2, \eta_{k+1}=-2\xi_{k}-\eta_{k}+\eta_{k}^2, $ and the polyhedron be given as $\mathcal{P}=\{(\xi,\eta)\,|\,\xi-\eta\leq -10, 2\xi-\eta\leq 10, \xi-2\eta\leq-20\}.$
\end{example}

We first directly show that $\mathcal{P}$ is an invariant set for the discrete system, i.e., $(\xi_{k+1},\eta_{k+1})\in\mathcal{P}$ for all $(\xi_k,\eta_k)\in \mathcal{P}.$ For simplicity, we only prove the first constraint, i.e., $\xi_{k+1}-\eta_{k+1}\leq -10$.  In fact, we have
$\xi_{k+1}-\eta_{k+1}=-\xi_k^2-\eta_k^2+\xi_k+3\eta_k=-\xi_k^2-(\eta_k-2.5)^2+\xi_k-2\eta_k+6.25
\leq \xi_k-2\eta_k+6.25\leq -20+6.25\leq -10.$
The other two constraints can be proved in a similar manner. On the other hand,  one can show that the assumption in Theorem \ref{ch5:polythm} is satisfied for this example. Then we can find a suitable $H\geq0$ such that condition (\ref{slede911}) holds. One can easily verify that $H=[0,0,1;0,0,0;1,0,1]$ satisfies condition (\ref{slede911}). Then  according to Theorem \ref{ch5:polythm}, we have that $\mathcal{P}$ is an invariant set for the discrete system.\\

We now consider an invariance condition for ellipsoids for the discrete  system (\ref{slnlsdis}).

\begin{theorem}\label{ch5:ellipthm}
Let  an ellipsoid   $\mathcal{E}=\{x\,|\,x^TQx\leq 1\}$, where $Q\in \mathbb{R}^{n\times n}$ and $Q\succ0$, and the  discrete system be given as in   (\ref{slnlsdis}).  Assume that $(f_d(x))^TQf_d(x)$ is a concave function.  Then $\mathcal{E}$ is an invariant
set for the discrete  system (\ref{slnlsdis})  if and only if there exists a $\beta\geq0$, such that
\begin{equation}\label{ch5:ellip3}
\beta x^TQx-(f_d(x))^TQf_d(x)\geq \beta-1, ~~~\text{ for all } x\in \mathbb{R}^n.
\end{equation}
\end{theorem}

\begin{proof}
The ellipsoid $\mathcal{E}$ is an invariant set for the discrete system if and only if $\mathcal{E}\subseteq\mathcal{E}',$ where $\mathcal{E}'=\{x\,|\,(f_d(x))^TQf_d(x)\leq1\}.$ We also note that $\mathcal{E}\subseteq\mathcal{E}'$ is equivalent to $(\mathbb{R}^n\setminus\mathcal{E}')\cap\mathcal{E}=\emptyset,$ i.e., the inequality system $1-(f_d(x))^TQf_d(x)<0$ and $x^TQx-1\leq0$ has no solution.
Since $(f_d(x))^TQf_d(x)$  is a concave function,
 we have that $1-(f_d(x))^TQf_d(x)$ is a convex function. Note that $x^TQx-1$ is also a convex function, according to Theorem \ref{nlf}, there exists a $\beta\geq0$, such that
$$
-(f_d(x))^TQf_d(x)+1 +\beta(x^TQx-1)\geq0, ~~~\text{ for all } x\in \mathbb{R}^n,
$$
which is the same as (\ref{ch5:ellip3}).
\end{proof}

\begin{remark}
If we choose $x=0$ in condition (\ref{ch5:ellip3}), then we have $\beta\leq 1-(f_d(0))^TQf_d(0),$ which can be considered as an upper bound of $\beta.$
\end{remark}

Similarly, we present the following optimization problem which is equivalent to condition  (\ref{ch5:ellip3}).
\begin{remark}
Consider the following  optimization problem
\begin{equation}\label{slvn311}
\max_{\beta\geq0}\min_{x\in\mathbb{R}^n}\{\beta x^TQx-(f_d(x))^TQf_d(x)- \beta+1\}.
\end{equation}
If the optimal objective value of   optimization problem (\ref{slvn311}) is nonnegative, then condition (\ref{ch5:ellip3}) holds.
\end{remark}

If the  system in Theorem \ref{ch5:ellipthm} is a linear dynamical system, then we  have the following corollary, which is an invariance condition of ellipsoids for linear system. Note that Corollary \ref{ch5:cord2}  can also be referred to \cite{song1}.

\begin{corollary}\label{ch5:cord2}
\emph{\cite{song1}}
Let  an ellipsoid  $\mathcal{E}=\{x\,|\,x^TQx\leq 1\}$, where $Q\in \mathbb{R}^{n\times n}$ and $Q\succ0$, and  a linear discrete system be given as in   (\ref{slnlsdis}), i.e., $f_d(x)=A_dx$, where $A_d\in \mathbb{R}^{n\times n}$. Then $\mathcal{E}$ is an invariant set for the discrete system (\ref{slnlsdis}), if and only if there exists a $\beta\in [0,1]$, such that
$
A_d^TQA_d-\beta Q\preceq0.
$
\end{corollary}
\begin{proof}
According to Theorem \ref{ch5:ellipthm}, we have that there exists a $\beta\geq0$, such that
\begin{equation}\label{eq:condt}
x^T(\beta Q-A_d^TQA_d)x\geq \beta-1,  \text{ for all }  x\in \mathbb{R}^n.
\end{equation}
If we choose $x=0$, then we have $\beta\leq1.$  Assume that $A_d^TQA_d-\beta Q\not\preceq0$, then there exists a negative eigenvalue $\lambda$ and the corresponding eigenvector $x^*\neq0$ such that $(\beta Q-A_d^TQA_d)x^*=\lambda x^*$, where $\lambda<0.$ Let  $y^*=\alpha x^*$ with $\alpha<\sqrt{\tfrac{{\beta-1}}{\lambda}}\tfrac{1}{\|x^*\|}$, then we have  $(y^*)^T(\beta Q-A_d^TQA_d)y^*<\beta-1$, which contradicts (\ref{eq:condt}). Thus, we have $A_d^TQA_d-\beta Q\preceq0.$
\end{proof}

Observe that parameter $\beta$ presented in Corollary \ref{ch5:cord2} can be eliminated. In fact, one can show that $A_d^TQA_d-\beta Q\preceq 0$ for $\beta\in[0,1]$ and $Q\succ0$ is equivalent to  $A_d^TQA_d-Q\preceq0,$ see \cite{song1}.

\textbf{In Theorem \ref{ch5:ellipthm}, we have the condition that $(f_d(x))^TQf_d(x)$ is a concave function. In fact, this is equivalent to verify if the Hessian of  $(f_d(x))^TQf_d(x)$ is negative semidefinite for all $x\in\mathbb{R}^n.$}

We now give an example to illustrate Theorem \ref{ch5:ellipthm}.

\begin{example}
Let the discrete system be $\xi_{k+1}=\tfrac{\sqrt{\xi_k+\eta_k}}{2}, \eta_{k+1}=\tfrac{\sqrt{\xi_k-3\eta_k}}{2},$ and the ellipsoid be given as $\mathcal{E}=\{(\xi,\eta)\,|\,\xi^2+\eta^2\leq1\}.$
\end{example}

For any $(\xi_k,\eta_k)\in \mathcal{E}$, we have $\xi_{k+1}^2+\eta_{k+1}^2=\frac{\xi_k-\eta_k}{2}\leq\tfrac{\sqrt{2}}{2}\sqrt{\xi_k^2+\eta_k^2}<1,$ which shows that $\mathcal{E}$ is an invariant set for the discrete system. On the other hand, let $f(x)=(f_1(x),f_2(x))^T=(\tfrac{\sqrt{\xi_k+\eta_k}}{2}, \tfrac{\sqrt{\xi_k-3\eta_k}}{2})^T$ and $Q=[1,0; 0,1]$. Then we have that $f(x)^TQf(x)$ is a concave function. If we choose $\beta=\tfrac{1}{4}$, then condition (\ref{ch5:ellip3}) yields $(\xi_k-1)^2+(\eta_k-1)^2+1\geq0$ for any $(\xi_k,\eta_k)\in \mathbb{R}^2.$ This, according to Theorem \ref{ch5:ellipthm},  also shows that  $\mathcal{E}$ is an invariant set for the discrete system.\\

We now consider invariance conditions for more general convex sets for discrete system (\ref{slnlsdis}). Let a  convex set be given as:
\begin{equation}\label{1zhcov1}
\mathcal{S}=\{x\in \mathbb{R}^n\,|\,g(x)\leq0\},
\end{equation}
where $g:\mathbb{R}^n\rightarrow \mathbb{R}$ is a convex function.
Then we have the following theorem, which gives invariance condition for the convex set (\ref{1zhcov1}) for discrete  system (\ref{slnlsdis}).

\begin{theorem}\label{zhthm11}
 Let the convex set $\mathcal{S}$ be given as in  (\ref{1zhcov1}),  and the discrete system be given as in (\ref{slnlsdis}). Assume that there exists $x^0\in \mathbb{R}^n$ such that $g(x)<0$, and that  $g(f_d(x))$ is a concave function. Then $\mathcal{S}$
is an invariant set for the discrete system  if and
only if there exists an $\alpha\geq0$, such that
\begin{equation}\label{zheq1211}
\alpha g(x)-g(f_d(x))\geq0, ~~\text{ for all } x\in \mathbb{R}^n.
\end{equation}
Moreover, if $g(x)$ and $g(f_d(x))$ are quadratic functions, then the assumption that $g(f_d(x))$ is a concave function is not required.
\end{theorem}

\begin{proof}
The major tool used in this proof is the convex Farkas Lemma, i.e., Theorem \ref{nlf}.
Note that to ensure $\mathcal{S}$ is an invariant set for the discrete system, we need to prove
$
\mathcal{S}\subseteq\mathcal{S}'=\{x\,|\,g(f_d(x))\leq0\},
$
i.e., $(\mathbb{R}^n\setminus \mathcal{S}')\cap \mathcal{S}=\emptyset.$ Then the following inequality system has no solution:
$$
-g(f_d(x))<0, ~~~g(x)\leq 0.
$$
According to Theorem \ref{nlf}, there exists an $\alpha\geq0$, such that
$$
-g(f_d(x))+\alpha g(x)\geq0, ~~~\text{ for } x\in \mathbb{R}^n,
$$
which is the same as (\ref{zheq1211}).
For the case of quadratic functions, we can use a similar argument and the
\emph{S}-Lemma to prove the last statement.
\end{proof}

\begin{remark}
The set $\mathcal{S}$ given as in (\ref{1zhcov1}) is represented by only a single convex function. In fact, the first statement in Theorem \ref{zhthm11} can be easily extended to the set which is presented by several convex functions, e.g., polyhedral sets.
\end{remark}

 The first statement in Theorem \ref{zhthm11} requires  that $g(f_d(x))$ is a concave function given that $g(x)$ is a convex function. Let us consider $x$ defined in a one dimensional space as an example\footnote{The example uses the following theorem: If $\tilde{g}(x)$ is a nondecreasing function, and $\tilde{f}(x)$ is a convex function, then $\tilde{g}(\tilde{f}(x))$ is a convex function.} to illustrate this case is indeed possbile.  Since $f_d(x)$ is a convex function, we have $f''_d(x)\geq0$ for all $x\in \mathbb{R}.$ For simplicity, we denote $h(x)=-g(f_d(x)).$ Then we have 
 \begin{equation} 
 h''(x)=-g''(f_d(x))(f_d(x))^2-g'(f_d(x))f_d''(x).
 \end{equation} 
 If $h''(x)>0$ for all $x\in \mathbb{R}$, then $h(x)$ is a convex function, i.e., $g(f_d(x))$ is a concave function. We now find a sufficient condition such that $h'(x)>0$  for all $x\in \mathbb{R}$.  Assume that $g(x)$ is a decreasing convex nonlinear function and $g(x)$ has no lower bound, we have  $g'(x)<0$ and $g''(x)>0$ for all $x\in \mathbb{R}.$ Assume $f_d(x)$ is a concave function, we have $f_d''(x)<0.$ This yields $-\tfrac{g'(f_d(x))}{g_d''(f(x))}>0\geq\tfrac{(f'(x))^2}{f''(x)}$, i.e., $h''(x)>0.$

\begin{remark}
Consider the following optimization problem:
\begin{equation}\label{svop311}
\max_{\alpha\geq0}\min_{x\in\mathbb{R}^n}\{\alpha g(x)-g(f_d(x))\}.
\end{equation}
If the optimal objective value of  optimization problem (\ref{svop311}) is nonnegative, we can claim that condition (\ref{zheq1211}) holds.
\end{remark}

Thus far we have three ``max-min" optimization problems shown as in  (\ref{sdopt351}), (\ref{slvn311}),  and (\ref{svop311}).  It is usually not easy to solve a ``max-min" problem.  In fact, these three problems can be transformed into a nonlinear optimization problem. Here we consider  (\ref{svop311}) as an example to illustrate this idea.  From here, we assume that $g(x)$ in (\ref{1zhcov1}) is continuously differentiable. 

\begin{theorem}\label{ch5:minmax}
Optimization problem (\ref{svop311}) is equivalent to the nonlinear optimization problem
\begin{equation}\label{optd3431}
\max_{x,\alpha}\{\alpha g(x)-g(f_d(x))\,|\,\alpha\nabla_xg(x)-\nabla_xg(f_d(x))=0,\alpha\geq0\}.
\end{equation}
\end{theorem}
\begin{proof}
Since $\alpha\geq0$, and the functions $g(x)$ and $-g(f_d(x))$ are both convex functions, we have that $\alpha g(x)-g(f_d(x))$ is also a convex function. Also, for $\alpha\geq0,$ the optimization problem
\begin{equation}\label{optd31}
\min_{x\in\mathbb{R}^n}\{\alpha g(x)-g(f_d(x))\},
\end{equation}
is a convex optimization problem in $\mathbb{R}^n$, thus problem (\ref{optd31}) has a Wolfe dual, see, e.g., \cite{terlakynlp, wolfe} given as follows:
\begin{equation}
\max_{x\in\mathbb{R}^n}\{\alpha g(x)-g(f_d(x))\,|\,\alpha\nabla_xg(x)-\nabla_x(g(f_d(x)))=0\}.
\end{equation}
Consequently, problem (\ref{svop311}) is equivalent to the nonlinear optimization problem  (\ref{optd3431}).
\end{proof}


\begin{remark}
One can use a proof similar to the one presented in Theorem \ref{ch5:minmax} to derive equivalent nonlinear optimization problems for the optimization problems  presented in (\ref{sdopt351}) and (\ref{slvn311}).
\end{remark}


We now consider an alternative way to investigate  invariance conditions for discrete  systems. The following lemma is easy to prove.
\begin{lemma}\label{lemd35}
Let $\phi(x),\psi(x):\mathbb{R}^n\rightarrow\mathbb{R}.$ The following two statements are equivalent:
\begin{itemize}
\item The inequality system $\phi(x)\leq0, \psi(x)>0$ has no solution.
\item The optimal objective value of the optimization problem 
\begin{equation}
\max\{\epsilon\,|\,\phi(x)\leq0, -\psi(x)+\epsilon\leq0\}
\end{equation} 
is nonpositive.
\end{itemize}
\end{lemma}

According to Lemma \ref{lemd35}, we have the following lemma. 
\begin{lemma}\label{lems3}
Let the discrete  system be given as in (\ref{slnlsdis}). Let  $\mathcal{S}_1=\{x\in \mathbb{R}^n\,|\,\phi(x)\leq0\}$ and $\mathcal{S}_2=\{x\in \mathbb{R}^n\,|\,\psi(x)\leq0\}$ be two closed sets\footnote{It is not necessary to assume that that the two sets are convex sets.}, where $\phi(x),\psi(x):\mathbb{R}^n\rightarrow\mathbb{R}.$ Then $x\in \mathcal{S}_1$ implies $f_d(x)\in\mathcal{S}_2$ if and only if  the optimal objective value of the following optimization problem
\begin{equation}\label{eqz0rdd}
\max\{\epsilon\,|\,\phi(x)\leq0, -\psi(f_d(x))+\epsilon\leq0\},
\end{equation}
is nonpositive.
\end{lemma}
\begin{proof}
We have that $x\in \mathcal{S}_1$ implies $f_d(x)\in \mathcal{S}_2$ if and only if $\mathcal{S}_1\subseteq\tilde{\mathcal{S}_2}=\{x\,|\,\psi(f_d(x))\leq0\}.$ This is equivalent to $(\mathbb{R}^n\setminus \tilde{\mathcal{S}}_2)\cap \mathcal{S}_1=\emptyset,$ i.e., the systems $\phi(x)\leq0$ and $\psi(f_d(x))>0$ have no solution. Then, according to Lemma \ref{lemd35}, the lemma is immediate.
\end{proof}
According to Lemma \ref{lems3}, we have the following theorem. 
\begin{theorem}
Let the discrete  system be given as in (\ref{slnlsdis}), and let  $\bar{\mathcal{S}}=\{x\in \mathbb{R}^n\,|\, \phi(x)\leq 0\}$ be a set, where $\phi(x):\mathbb{R}^n\rightarrow \mathbb{R}$. Then $\bar{\mathcal{S}}$ is an invariant set for the discrete system if and only if the optimal objective value of the following optimization problem 
\begin{equation}
\max\{\epsilon\,|\, \phi(x)\leq 0, -\phi(f_d(x))+\epsilon \leq 0\}
\end{equation}
is nonpositive. 

\end{theorem}
\begin{proof}
The set is an invariant set for the discrete system if and only if $\bar{\mathcal{S}}\subseteq\tilde{\mathcal{S}}=\{x\,|\,\phi(f_d(x))\leq 0\}.$ According to Lemma \ref{lems3}, the theorem is immediate. 
\end{proof}

\section{Invariance Conditions for Continuous  Systems}
In this section, we consider invariance conditions  for continuous  systems in the form of (\ref{slnlscon}). For  discrete systems, in Section \ref{dis}, we  transformed the invariance conditions into ``max-min" optimization problems, which were later proved to be equivalent to traditional nonlinear optimization problems. For the continuous systems, we transform the invariance conditions into nonlinear optimization problems, too.

First, we consider an invariance condition for  continuous system (\ref{slnlscon}) and for polyhedral sets $\mathcal{P}=\{x\,\,Gx\leq b\}$, where $G\in \mathbb{R}^{m\times n}$ and $b\in \mathbb{R}^m$. For simplicity we assume that the origin is in the interior of the polyhedral set, thus we have $\mathcal{P}=\{x\in\mathbb{R}^n\,|\,Gx\leq b\}=\{x\in\mathbb{R}^n\,|\,g_i^Tx\leq b_i, i\in\mathcal{I}(m)\}$, where $b>0$.

\begin{theorem}\label{labdd}
Let a polyhedral set be given as $\mathcal{P}=\{x\in\mathbb{R}^n\,|\,g_i^Tx\leq b_i, i\in\mathcal{I}(m)\}$, where $b>0,$ and let $\mathcal{P}^i=\{ x\in\mathcal{P}\,|\,g_i^Tx=b_i \}$ for $i\in\mathcal{I}(m).$ Then $\mathcal{P}$ is an invariant set for the continuous system (\ref{slnlscon}) if and only if for all $i\in\mathcal{I}(m)$
\begin{equation}\label{dagagd}
g_i^Tf_c(x)\leq0 \text{ holds for all } x\in\mathcal{P}^i.
\end{equation}
\end{theorem}

\begin{proof}
Let $x\in\partial\mathcal{P}$. Then we have that $x$ is in the relative interior of a face, on the relative boundary, or a vertex of $\mathcal{P}$. There exists a maximal index set $\mathcal{I}_x$ such that $x\in\cap_{i\in \mathcal{I}_x}\mathcal{P}^i.$
We note that $\mathcal{T_P}(x)=\{y\in\mathbb{R}^n\,|\,g_i^Ty\leq0, i\in \mathcal{I}_x\}$, then, according to Nagumo Theorem \ref{nagumo}, the theorem is immediate. 
\end{proof}

\begin{remark}
Let us assume  a polyhedral set $\mathcal{P}$ be given as in the statement of Theorem \ref{labdd}. Consider the following $m$ optimization problems:
\begin{equation}\label{dagad}
\max\{g_i^Tf_c(x)\,|\, g_i^Tx=b_i \text{ and } x\in \mathcal{P}\}, i\in \mathcal{I}(m).
\end{equation}
If the optimal objective values of  all the $m$ optimization problems in (\ref{dagad}) are nonpositive, then we can claim that (\ref{dagagd}) holds. 
\end{remark}

\textbf{Clearly, when $g_i^Tf_c(x)$ is a concave function, problem (\ref{dagad}) is a convex problem, which is very easy to solve by using optimization solver. Otherwise, this problem is a nonconvex problem, which may need special nonlinear algorithms to solve. }

Invariance conditions for  continuous system (\ref{slnlscon}) and  for ellipsoids or Lorenz cones  is presented  in the following theorem.
\begin{theorem}\label{ch1:lemma4}
 Let  the ellipsoid $\mathcal{E}=\{x\,|\,x^TQx\leq 1\}$, where $Q\in \mathbb{R}^{n\times n}$ and $Q\succ0$, and  the continuous system be given as in   (\ref{slnlscon}). Then  $\mathcal{E}$ 
is an invariant set for the continuous  system (\ref{slnlscon}) if and
only if
\begin{equation}\label{no1}
(f_c(x))^TQx\leq 0, \text{ for all } x\in \partial \mathcal{E}.
\end{equation}
\end{theorem}
\begin{proof}
Note that $\partial
\mathcal{E}=\{x\,|\,x^TQx=1\}$, thus the outer normal vector of $\mathcal{E}$ at $x\in \partial \mathcal{E}$ is $f_d(x)$. Then we have that the tangent cone at $x\in \partial\mathcal{E}$ is given as $\mathcal{T}_{\mathcal{E}}(x)=\{y\,|\, y^TQx\leq0\}$, thus this
theorem follows by the Nagumo Theorem \ref{nagumo}.
\end{proof}
Note that  Theorem \ref{ch1:lemma4} can be applied to  a Lorenz cone $\mathcal{C_L}$, see, e.g., \cite{song1}. 
\begin{remark}
Let us consider an ellipsoid $\mathcal{E}$ and  the following optimization problem:
\begin{equation}\label{slvn3}
\max \{ (f_c(x))^TQx\,|\, x^TQx=1\}.
\end{equation}
If the global optimal objective value of optimization problem (\ref{slvn3}) is nonpositive, then condition (\ref{no1}) holds.
\end{remark}

\textbf{We note that problem (\ref{slvn3}) is not a convex optimization problem since the feasible region  $\{x\,|\,x^TQx=1\}$ is not a convex region. Thus, nonconvex optimization algorithms are required to solve this problem. }

\begin{theorem}
 Let the convex set $\mathcal{S}$ be given as in (\ref{1zhcov1}) and let function $g(x)$ be continuously differentiable. Then $\mathcal{S}$
is an invariant set for the continuous system (\ref{slnlscon}) if and
only if
\begin{equation}\label{zheq25}
(\nabla g(x))^Tf_c(x)\leq0, ~~~~\text{ for all } x\in \partial \mathcal{S}.
\end{equation}
\end{theorem}
\begin{proof}
The outer normal vector at  $x\in \partial \mathcal{S}$ is  $\nabla g(x)$. Since $\mathcal{S}$ is a convex set, we have 
\begin{equation}
\mathcal{T_S}(x)=\{y\,|\,(\nabla g(x))^Ty\leq0\}.
\end{equation}
The proof is immediate by applying Nagumo's Theorem \ref{nagumo}.
\end{proof}
\begin{remark}
Consider the following optimization problem:
\begin{equation}\label{optd45d}
\max\{\alpha\,|\,\alpha=(\nabla g(x))^Tf_c(x), ~g(x)=0\}.
\end{equation}
If the optimal objective value of optimization problem (\ref{optd45d}) is nonpositive, then we can claim that condition (\ref{zheq25}) holds.
\end{remark}

\textbf{We note that when problem (\ref{optd45d}) is not a convex optimization, we may need nonconvex optimization algorithm to solve this problem. }

\section{General Results for Discrete Systems}
In Section \ref{dis},  invariance conditions for polyhedral sets, ellipsoids, and convex sets are presented under certain assumptions.  In this section, invariance conditions for these sets for discrete systems are presented without any assumption.  First let us consider polyhedral sets. 

\begin{theorem}\label{gen1}
Let  the polyhedron  $\mathcal{P}=\{x\,|\, Gx\leq b\},$ where $G\in \mathbb{R}^{m\times n}$ and $b\in \mathbb{R}^m,$ and  the discrete system be given as in   (\ref{slnlsdis}).  Then $\mathcal{P}$ is an invariant set for the discrete
system (\ref{slnlsdis}) if and only if there exists a matrix $H\geq0$,
such that
\begin{equation}\label{slede912}
HGx-Gf_d(x)\geq Hb-b,~~ \text{ for all } x\in \mathcal{P}.
\end{equation}
\end{theorem}
\begin{proof}
\emph{Sufficiency:} Condition (\ref{slede912}) can be reformulated as $b-Gf_d(x)\geq H(b-Gx)$, where $x\in\mathcal{P},$ i.e., $b-Gx\geq0$. Since $H\geq0$, we have $b-Gf_d(x)\geq0$, i.e., $f_d(x)\in \mathcal{P}$ for all $x\in \mathcal{P}.$ Thus $\mathcal{P}$ is an invariant set for the discrete system.
\emph{Necessity:} Assume $\mathcal{P}$ is an invariant set for the discrete system. Then for any $x_k\in \mathcal{P}$, we have $x_{k+1}=f_d(x_k)\in \mathcal{P},$  i.e., we have that $b-Gx\geq0$ implies $b-Gf_d(x)\geq0$. Thus, we can choose $H=0.$
\end{proof}

Note that the difference between conditions (\ref{slede911})  and (\ref{slede912}) is that the same inequality holds, for $x\in\mathbb{R}^n$ in (\ref{slede911}), and for $x\in \mathcal{P}$ in (\ref{slede912}), respectively. Similarly, we also have the following remark. 

\begin{remark}\label{remgen1}
Consider the following  $m$ optimization problems
\begin{equation}\label{sdopt355}
\max_{H_i\geq0}\min_{x}\{H_i^TGx-G_i^Tf_d(x)-H_i^Tb+b_i\,|\,Gx\leq b\}~~~ i\in\mathcal{I}(m).
\end{equation}
If the global optimal objective values of the $m$ optimization problems in (\ref{sdopt355}) are all  nonnegative, then condition (\ref{slede912}) holds.

\end{remark} 

We now present an invariance condition for ellipsoids for discrete systems. In this invariance condition, for ellipsoids no assumption is needed. 

\begin{theorem}\label{ellipthm12}
Let  the ellipsoid $\mathcal{E}=\{x\,|\,x^TQx\leq 1\}$, where $Q\in \mathbb{R}^{n\times n}$ and $Q\succ0$, and let the discrete system be given as in   (\ref{slnlsdis}).  Then $\mathcal{E}$ is an invariant
set for the discrete  system  if and only if there exists a $\beta\geq0$, such that
\begin{equation}\label{ellip312}
\beta x^TQx-(f_d(x))^TQf_d(x)\geq \beta-1, ~~~\text{ for all } x\in \mathcal{E}.
\end{equation}
\end{theorem}
\begin{proof}
\emph{Sufficiency:} Condition (\ref{ellip312}) can be reformulated as $1-(f_d(x))^TQf_d(x)\geq \beta(1-x^TQx)$, where $x\in \mathcal{E}$. Thus we have $1-(f_d(x))^TQf_d(x)\geq0$, i.e., $f_d(x)\in \mathcal{E}$. Thus $\mathcal{E}$ is an invariant set for the discrete system.
\emph{Necessity:} It is immediate by choosing $\beta=0.$
\end{proof}

\begin{remark}\label{remgen2}
Consider the following  optimization problem
\begin{equation}\label{slvn315}
\max_{\beta\geq0}\min_{x}\{\beta x^TQx-(f_d(x))^TQf_d(x)- \beta+1\,|\,x^TQx\leq 1\}.
\end{equation}
If the optimal objective value of   optimization problem (\ref{slvn315}) is nonnegative,  then  condition (\ref{ellip312}) holds.

\end{remark}

We now present an invariance condition for convex sets and for discrete systems. In this invariance condition, no assumption is needed for convex sets. 

\begin{theorem}\label{zhthm12}
Let the convex set $\mathcal{S}$ be given as in  (\ref{1zhcov1}) and let the discrete system be given as in (\ref{slnlsdis}). Then $\mathcal{S}$
is an invariant set for the discrete system  if and
only if there exists an $\alpha\geq0$, such that
\begin{equation}\label{zheq1212}
\alpha g(x)-g(f_d(x))\geq0, ~~~~\text{ for all } x\in \mathcal{S}.
\end{equation}
\end{theorem}
\begin{proof}
\emph{Sufficiency:} Condition (\ref{zheq1212}) can be reformulated as $\alpha g(x)\geq g(f_d(x))$, where $x\in \mathcal{S}$, i.e., $g(x)\leq0$. According to $\alpha\geq0,$ we have $g(f_d(x))\leq0$, i.e., $f_d(x)\in \mathcal{S}$. Thus $\mathcal{S}$ is an invariant set for the discrete system.
\emph{Necessity:} It is immediate by choosing $\alpha=0.$
\end{proof}

\begin{remark}\label{remgen3}
Consider the following optimization problem:
\begin{equation}\label{svop315}
\max_{\alpha\geq0}\min_{x\in\mathbb{R}^n}\{\alpha g(x)-g(f_d(x))\}.
\end{equation}
If the optimal objective value of  optimization problem (\ref{svop315}) is nonnegative, then condition (\ref{zheq1212}) holds.
\end{remark}

We note that  there are no assumptions in Theorem \ref{gen1}, \ref{ellipthm12}, and \ref{zhthm12}, which means we cannot use the Wolfe duality theory. Thus we cannot transform the ``max-min" optimization problems in Remark \ref{remgen1}, \ref{remgen2}, and \ref{remgen3} into nonlinear maximization problems. The absence of convexity assumptions makes the theorems more applicable, however the nonlinear feasibility problems 
(\ref{slede912}), (\ref{ellip312}), and (\ref{zheq1212}) are nonconvex,  thus their verification is significantly harder than solving convex feasibility problems. \textbf{As we pointed out in introduction that there are very few papers studying invariance conditions for nonlinear systems, the  nonlinear feasibility problems 
(\ref{slede912}), (\ref{ellip312}), and (\ref{zheq1212}) provide us a novel perspective to consider invariance conditions.  They also bring the possibility of applying state-of-the-art optimization algorithms to solve invariance problems.}

\section{Conclusions}
In this paper we derived invariance conditions for some classical sets for nonlinear dynamical systems by utilizing  a methodology analogous to the one presented in \cite{song1}. This is motivated by the fact that most problems in the real world are modeled by nonlinear dynamical systems, because they often show nonlinear characteristics. 
The Theorems of Alternatives, i.e., the nonlinear Farkas lemma and the \emph{S}-lemma, together with 
Nagumo's Theorem  are our main tools to derive invariance conditions for discrete and continuous systems. We derive the invariance conditions for these classic sets for nonlinear systems with some, and without any conditions.
We also propose an optimization problem for each invariance condition. Then to verify the invariance condition is equivalent to solve the corresponding optimization problem. These invariance conditions provide potential ways to design algorithms to establish invariant sets for a system. The introduction of the associated optimization problem opens new avenues to use advanced optimization algorithms and software to solve invariant set problems. 

\section*{Acknowledgments}
{This research is supported by a Start-up grant of Lehigh University, and by
TAMOP-4.2.2.A-11/1KONV-2012-0012: Basic research for the development of
hybrid and electric vehicles. The TAMOP Project is supported by the European Union
and co-financed by the European Regional Development Fund.}

\bibliographystyle{plain}
\bibliography{generalref}

\section{Appendix}
\begin{theorem} \emph{\cite{bits2, song1}}
Let  $\mathcal{P}=\{x\,|\, Gx\leq b\}$ be a polyhedron, where $G\in \mathbb{R}^{m\times n}$ and $b\in \mathbb{R}^m.$ Let the discrete system, given as in   (\ref{slnlsdis}), be linear, i.e., $f(x)=Ax$. Then $\mathcal{P}$ is an invariant set for the discrete
system (\ref{slnlsdis}) if and only if there exists a matrix $H\geq0$, such that $HG=GA$ and $Hb\leq b.$
\end{theorem}
\begin{proof}
We have that $\mathcal{P}$ is an invariant set for the linear system if and only if the optimal objective values of the following $n$ linear optimization problems are all nonnegative: 
\begin{equation}\label{eqlab3}
\min \{b_i-G_i^TAx\,|\,Gx\leq b\}~~~ i\in\mathcal{I}(m).
\end{equation}
Problems (\ref{eqlab3}) are equivalent to 
\begin{equation}\label{eqlab4}
-b_i+\max G_i^TAx, \emph{ s.t. } Gx\leq b~~~ i\in\mathcal{I}(m).
\end{equation}
The duals of these linear optimization problems presented in (\ref{eqlab4}) are for all $i\in \mathcal{I}(m)$
\begin{equation}\label{kktod2}
\begin{split}
-b_i+\min &~ b^TH_i\\
\text{s.t. }&G^TH_i=A^TG_i\\
&H_i\geq0.
\end{split}
\end{equation}
Due to the Strong Duality Theorem of linear optimization, see, e.g., \cite{Roos}, the primal and dual objective function values are equal at optimal solutions, thus $G_i^TAx=b^TH_i.$ As the optimal value of (\ref{eqlab3}) is nonnegative for all $i\in\mathcal{I}(m)$, we have $b_i-b^TH_i\geq0$. Thus $b\geq Hb$. The proof is complete. 
%
%
%
\end{proof}

\end{document}